\documentclass[12pt]{elsarticle}
\makeatletter
\def\ps@pprintTitle{%
  \let\@oddhead\@empty
  \let\@evenhead\@empty
  \let\@oddfoot\@empty
  \let\@evenfoot\@oddfoot
}
\makeatother

\usepackage{amsmath,amsthm,amsfonts,amssymb}
\usepackage{comment}
\usepackage{xcolor}

\usepackage[shortlabels]{enumitem}
\setlist[enumerate,1]{label={\upshape(\roman*)}}
\usepackage{hyperref}
\usepackage{mathrsfs}

\hypersetup{
    colorlinks,
    linkcolor={red!50!black},
    citecolor={blue!50!black},
    urlcolor={blue!80!black}
}

\usepackage{mathtools}
\newcommand{\sV}[2]{{
	\setlength{\arraycolsep}{2pt}
	\renewcommand{\arraystretch}{0.8}
	\left[\begin{array}{ccc} #1 \\ #2 \end{array}\right]
}}

\newtheorem{thm}{Theorem}[section]
\newtheorem{lem}[thm]{Lemma}
\newtheorem{cor}[thm]{Corollary}
\newtheorem{prop}[thm]{Proposition}
\theoremstyle{definition}
\newtheorem{dfn}[thm]{Definition}
\newtheorem{remark}[thm]{Remark}
\newtheorem{exmp}[thm]{Example}

\newcommand{\tA}{\widetilde{\mathsf{A}}}
\newcommand{\C}{\mathbb{C}}

\newcommand{\allone}{\mathbf{1}}
\newcommand{\allzero}{\mathbf{0}}

\newcommand{\cX}{\mathcal{X}}
\newcommand{\cY}{\mathcal{Y}}

\newcommand{\Trel}{T}
\newcommand{\Srel}{S}
\newcommand{\Rrel}{R}

\newcommand{\Pmat}{\mathsf{P}}
\newcommand{\Qmat}{\mathsf{Q}}
\newcommand{\Imat}{\mathsf{I}}
\newcommand{\Emat}{\mathsf{E}}
\newcommand{\Jmat}{\mathsf{J}}
\newcommand{\Amat}{\mathsf{A}}
\newcommand{\Rmat}{\mathsf{B}}
\newcommand{\Bmat}{\mathsf{S}}
\newcommand{\Cmat}{\mathsf{T}}
\newcommand{\Smat}{\mathsf{S}}
\newcommand{\Tmat}{\mathsf{T}}
\newcommand{\Mmat}{\mathsf{M}}
\newcommand{\Perm}{\mathsf{P}}

\title{Roux schemes which carry association schemes locally}

\begin{document}
\begin{frontmatter}
\author[Memphis]{Alexander L. Gavrilyuk}
\ead{a.gavrilyuk@memphis.edu}
\author[Tohoku]{Jesse Lansdown\corref{cor1}\fnref{label2}}
\ead{jesse.lansdown@universityofgalway.ie}
\author[Tohoku]{Akihiro Munemasa}
\ead{munemasa@tohoku.ac.jp}
\author[NDAJ]{Sho Suda}
\ead{ssuda@nda.ac.jp}

\fntext[label2]{Current affiliation: School of Mathematical and Statistical Sciences, University of Galway, Galway, Ireland}
\cortext[cor1]{Corresponding author}

\address[Memphis]{Department of Mathematical Sciences, University of Memphis, USA}
\address[Tohoku]{Graduate School of Information Sciences, Tohoku University, Japan}
\address[NDAJ]{Department of Mathematics, National Defense Academy of Japan, Japan}

\begin{abstract}
A roux scheme is an association scheme formed from a special ``roux'' matrix and the regular permutation representation of an associated group.
They were introduced by Iverson and Mixon in \cite{IversonMixon} for their connection to equiangular tight frames and doubly transitive lines.
We show how roux matrices can be produced from association schemes and characterise roux schemes for which the neighbourhood of a vertex induces an association scheme possessing the same number of relations as the thin radical.
An important example arises from the $64$ equiangular lines in $\C^8$ constructed by Hoggar \cite{Hoggar} which we prove is unique (determined by its parameters up to isomorphism).
 We also characterise roux schemes by their eigenmatrices and 
 provide new families of roux schemes using our construction.
\end{abstract}


\end{frontmatter}

\section{Introduction}\label{sec:intro}
Equiangular lines (and equiangular tight frames in particular) are optimal packings of lines in complex space and as a result they have important applications in fields such as signal processing, quantum information theory, and compressed sensing. A sequence $\mathcal{L} = \{\ell_i\}_{i=1}^m$ of lines through the origin in $\C^d$ is called \emph{equiangular} if there are unit norm vectors $\varphi_i \in \ell_i$ and a constant $\alpha$ such that
\[
|\langle \varphi_i, \varphi_j \rangle | = \alpha
\]
for all $i \neq j$.
A set of equiangular lines in $\C^d$ has size at most $d^2$.
Furthermore, a set of unit vectors $\{\varphi_i\}_{i=1}^m$ is called an \emph{equiangular tight frame (or ETF)} if $\mathcal{L} =  \{\mathbb{C}\varphi_i\}_{i=1}^m$ is equiangular and the matrix $\mathsf{\Phi} = \left[ \varphi_1 \ldots \varphi_m \right] \in \C^{d \times m}$ satisfies $\mathsf{\Phi} \mathsf{\Phi}^* = \beta \mathsf{I}$ for some $\beta \in \mathbb{R}$.

An important example is the set of $64$ equiangular lines in $\C^8$ with $\alpha = \frac{1}{3}$ constructed by Hoggar \cite{Hoggar}, which we shall refer to as \emph{Hoggar's lines}.
Let $\Phi = \{\varphi_j\}_{j=1}^{64}$ be the unit norm vectors corresponding to Hoggar's lines (as given in \cite{Hoggar}) and let
\[
\Omega=\Phi\cup i \Phi \cup(-\Phi) \cup(-i \Phi).
\]  Then we can define a binary relation $\Rrel_j$ on $\Omega$ for $j=0,1,\ldots,7$ by 
\begin{equation*}
(\boldsymbol{x},\boldsymbol{y})\in \Rrel_j \Leftrightarrow  \langle \boldsymbol{x},\boldsymbol{y}\rangle=\alpha_j \quad (\boldsymbol{x},\boldsymbol{y}\in \Omega),
\end{equation*}
where 
$
\alpha_0=1,\alpha_1=i,\alpha_2=-1,\alpha_3=-i,
\alpha_4=1/3,\alpha_5=i/3,\alpha_6=-1/3,\alpha_7=-i/3.
$
Then the pair $\mathcal{H}=(\Omega,\{\Rrel_0,\Rrel_1,\ldots,\Rrel_7\})$ is a $7$-class association scheme, which we shall refer to as \emph{Hoggar's scheme} (and denote by $\mathcal{H}$) due to its dependence on Hoggar's lines.

The Gram matrix $(\langle \varphi_i, \varphi_j \rangle)_{ij}$ of an ETF is a scalar multiple of an orthogonal projection matrix, where the off-diagonal entries have the same modulus. Note that orthogonal projection matrices may be found in the Bose-Mesner algebra of an association scheme by taking any sum of the primitive idempotents. This fact, combined with the example of Hoggar's scheme, suggests a connection between ETFs and association schemes.

This connection was explored by Iverson and Mixon in \cite{IversonMixon} where they introduce roux schemes and roux matrices. The term ``roux'' were coined because (analagous to roux in cooking) they are formed when an otherwise thin association scheme is ``thickened''. 
\begin{dfn} \label{def:rouxmat}
A \emph{roux matrix} for an abelian group $G$ is an $n\times n$ matrix $\Rmat$ ($n>2$) with entries in $\mathbb{C}[G]$ such that:
\begin{enumerate}
        \item[(R1)] $\Rmat_{ii}=0$ for $i \in \{1, \ldots, n\}$, \label{cond:roux1}
        \item[(R2)] $\Rmat_{ij} \in G$ for $i, j \in \{1, \ldots, n\}$, $i \neq j$,\label{cond:roux2}
        \item[(R3)] $\Rmat_{ji}=(\Rmat_{ij})^{-1}$ for $i, j \in \{1, \ldots, n\}$, $i \neq j$,\label{cond:roux3}
        \item[(R4)] The matrices $\{g\mathsf{I}\}_{g \in G}$ and $\{g\Rmat\}_{g \in G}$ span an algebra. \label{cond:roux4}
    \end{enumerate}
\end{dfn}
Given a roux matrix  $\Rmat$ for some abelian group $G$ of order $r$, let $\lceil\cdot\rfloor\colon \mathbb{C}[G]^{n \times n} \to \mathbb{C}^{rn \times rn}$ denote the injective algebra homomorphism which applies the regular permutation representation of $G$ to each entry of the matrix. Then $\{\lceil g\Imat\rfloor\}_{g \in G}$ and $\{\lceil g\Rmat \rfloor\}_{g \in G}$ are the adjacency matrices of a commutative association scheme, and any scheme which can be obtained in this way is called a \emph{roux scheme}. Note that $\{\lceil g\Rmat \rfloor\}_{g \in G}$ represent thick relations, due to the assumption that $n>2$. As a result, a roux matrix can be seen as a compact representation of a roux scheme. Evaluating a roux matrix at a character of the corresponding group produces the signature matrix of an ETF. We refer to \cite{IversonMixon} and Figure 1 therein for more relations between roux matrices, roux schemes, and ETFs.

Roux schemes generalise the construction of Godsil and Hensel \cite{GodsilHensel} and so include regular abelian distance-regular antipodal covers of the complete graph (DRACKNs).  Observe that a roux scheme has $r$ relations of valencies $1$ and $r$ relations of valencey $n-1$, and so it is a pseudo-TI scheme \cite{ChenPonomarenko} (if it is also Schurian then it is a TI-scheme). The focus of \cite{IversonMixon} is primarily on doubly transitive lines, so the Schurian case of roux schemes is of central importance to them. These are shown to correspond to special group-subgroup pairs called Higman pairs by the authors of \cite{IversonMixon}, which are a special case of Gelfand pairs with additional conditions necessary to ensure the roux properties. Other than a construction involving conference matrices, the constructions of roux schemes in \cite{IversonMixon} and the sequel paper \cite{IversonMixon2} centre primarily on the use of Higman pairs, and hence are Schurian schemes.

In fact, strong symmetry conditions are an integral component of roux schemes, as demonstrated in the following lemma which characterises roux schemes.

\begin{lem}[{\cite[Lemma 2.2]{IversonMixon}}] \label{lem:RouxRadicalCharacterisation}
An association scheme is isomorphic to a roux scheme if and only if it is commutative and its thin radical acts regularly (by multiplication) on the other adjacency matrices, at least one of which is symmetric.
\end{lem}

Although the definition of a roux matrix is relatively simple, finding such matrices is highly non-trivial. This is due mostly to the difficulty in producing an algebra to satisfy property (R4). In this paper we consider the construction of roux matrices from association schemes. This is a natural consideration since there is already an algebra, namely the Bose-Mesner algebra, connected to the association scheme. 
In Section \ref{sec:construction} we provide such a construction of roux matrices from a group and an association scheme, determine necessary compatability conditions between the two, and provide specific constructions and examples illustrating this method.

In Section \ref{sec:EigenCharacterisation} we determine the eigenmatrices of roux schemes and show that the structure of the eigenmatrices characterises roux schemes.

By constructing roux matrices from association schemes, the original association scheme appears locally in the neighbourhood of a vertex of the resulting roux scheme. We show in Section \ref{sec:Decomp} how to decompose roux schemes and in particular that they always correspond to our construction from Section \ref{sec:construction} if they carry an association scheme locally about any vertex which has the same number of classes as the thin radical.

Finally in Section \ref{sec:Uniqueness} we consider Hoggar's scheme. One of our motivations was to see if there are other roux schemes with the same parameters as Hoggar’s scheme. This would result in new ETFs similar to Hoggar's lines. However, we show that Hoggar's scheme is unique, that is, any association scheme with the same parameters as Hoggar's scheme is isomorphic to Hoggar's scheme. We still hope that Hoggar's scheme might belong to a family of association schemes that fit within our frame work, resulting in ETFs in different dimensions.

\section{Background}

\subsection{Association schemes}
Let us recall some standard facts from the theory of association schemes (see \cite{BannaiIto}). Let $\Omega$ be a finite set of cardinality $n$, and let $\{\Rrel_i\}_{i=0}^d$ be a set of binary relations on $\Omega$. Define $\Rrel_i^\top = \{ (y, x) \mid (x,y) \in \Rrel_i\}$. Then $\cX = (\Omega, \{\Rrel_i\}_{i=0}^d)$ is an 
(commutative)
\emph{association scheme} 
of $d$ classes
if the following properties hold:
\begin{enumerate}
\item $\Rrel_0$ is the identity relation,
\item $\{\Rrel_i\}_{i=0}^d$ is a partition of $\Omega \times \Omega$,
\item $\Rrel_j^\top\in\{\Rrel_i\}_{i=0}^d$ ($0 \le j \le d$),
\item Given $0 \le i,j,k \le d$, the number
	$p_{ij}^k = |\{z \mid (x, z) \in \Rrel_i, (z, y) \in \Rrel_j\}|$
	does not depend on the choice of $(x, y) \in \Rrel_k$.
\item  $p_{ij}^k = p_{ji}^k$
 ($0 \le i,j \le d$).
\end{enumerate}
The non-negative integers $p_{ij}^k$ are called the \emph{intersection numbers}, which we also refer to collectively as the parameters of the association scheme. We will often refer to an association scheme simply as a ``scheme''.
A scheme is \emph{symmetric} 
if $\Rrel_j^\top=\Rrel_j$ for $0 \leq j \leq d$.

Let $\Amat_i$ denote the adjacency matrix of the relation $\Rrel_i$ for $i=0,1,\ldots,d$, then an association scheme can equivalently be described by the following conditions:
\begin{enumerate}
\item $\Amat_0=\Imat_n$, the identity matrix of size $n$,
\item $\sum_{i=0}^d \Amat_i = \Jmat$, the square all-one matrix of size $n$,
\item $\Amat_i^\top\in\{\Amat_0,\Amat_1,\ldots,\Amat_d\}$ ($0 \le i \le d$),
\item $\Amat_i\Amat_j=\sum_{k=0}^d p_{ij}^k\Amat_k$,
\item  $\Amat_i\Amat_j=\Amat_j\Amat_i$
 ($0 \le i,j \le d$).
\end{enumerate}
We refer to $\Amat_0, \ldots, \Amat_d$ as the \emph{adjacency matrices} and note that $\Amat_i^\top$ is the standard matrix transpose.
Throughout, we will use normal and sans serif font to show the connection between relations and adjacency matrices. For example we use normal font $\Trel$ and $\Srel$ to represent thin and thick relations respectively, while sans serif font $\Tmat$ and $\Smat$ is used for the corresponding adjacency matrices.

The \emph{Bose-Mesner algebra} (or \emph{adjacency algebra}) of the association scheme is the $(d+1)$-dimensional $\mathbb{C}$-algebra spanned by the adjacency matrices.
It is closed under both standard matrix multiplication and also entrywise multiplication (also called Schur or Hadamard multiplication, and denoted by $\circ$). There is a second basis for
the Bose-Mesner algebra
consisting of primitive idempotents, $\Emat_0, \ldots, \Emat_d$. These minimal idempotents obey a similar condition to the adjacency matrices, namely
\[
\Emat_i\circ \Emat_j=\frac{1}{n}\sum_{k=0}^d q_{ij}^k\Emat_k,
\]
for $0 \leq i,j \leq d$. The $q_{ij}^k$ are non-negative constants called the \emph{Krein parameters}. Moreover, there exists a nonsingular matrix $\mathsf{P}$ called the \emph{(first) eigenmatrix} such that $(\Amat_0, \Amat_1, \ldots, \Amat_d) = (\Emat_0, \Emat_1, \ldots, \Emat_d)\mathsf{P}$. The \emph{second eigenmatrix} $\mathsf{Q}$ is obtained 
by $\mathsf{Q}= n\mathsf{P}^{-1}$.

The \emph{valencies} of an association scheme are defined as $k_i = p_{ii}^0$ for $0\leq i \leq d$. We call a relation $\Rrel_i$ \emph{thin} if $k_i=1$. Note that for a thin relation, $\Amat_i$ is a permutation matrix, and so acts on the elements of 
the Bose-Mesner algebra
by normal matrix multiplication. In fact, the set of thin relations is a group, called the \emph{thin radical} \cite{Zieschang}.

For any vertex $z$ we define its \emph{neighbourhood} with respect to the relation $\Rrel_i$ by $\Rrel_i(z) :=\{x \in \Omega\colon (z, x) \in \Rrel_i\}$. Considering the restriction of each relation to this neighbourhood 
\[
\Rrel_j' := \{(x, y) \in \Rrel_j\colon x, y \in \Rrel_i(z)\},
\]
if $(\Rrel_i(z), \{R_j'\colon R_j' \neq \varnothing\})$ forms an association scheme then we refer to it as a \emph{local scheme} and say that $\cX$ \emph{carries an association scheme locally about the vertex $z$}.

\subsection{Triple intersection numbers} 
Triple intersection numbers and vanishing Krein parameters were first considered by Cameron, Goethals and Seidel for strongly regular graphs \cite{CameronEtAl1978}, by 
Coolsaet and Juri\v{s}i\'{c} for 
distance-regular graphs 
\cite{TripleKrein} and by Gavrilyuk, Vidali and Williford for symmetric association schemes \cite{Q-TripleKrein}. Here we develop the theory of triple intersection numbers for 
an arbitrary commutative (not necessarily symmetric) association scheme 
$(\Omega,\{\Rrel_0,\Rrel_1,\ldots,\Rrel_d\})$ and give Proposition \ref{prop:qijk=0} which generalises \cite[Theorem 3]{TripleKrein}. 

For a triple of vertices $u, v, w \in \Omega$ and integers $i$, $j$, $k$ ($0 \le i, j, k \le d$)
we denote by $\sV{u & v & w}{i & j & k}$
(or simply $[i\ j\ k]$ when it is clear
which triple $(u,v,w)$ we have in mind)
the number of vertices $x \in \Omega$ such that
$(u, x) \in \Rrel_i$, $(v, x) \in \Rrel_j$ and $(w, x) \in \Rrel_k$.
We call these numbers {\em triple intersection numbers} 
(with respect to $u,v,w$).
A scheme is said to be
\emph{triply regular} if the triple 
intersection numbers $[i\ j\ k]$ 
only depend on $i,j,k$ and 
the relations containing the pairs $(u, v)$, $(v,w)$, and $(u,w)$
but not on the particular choice of vertices $u,v,w$.

Similar to the symmetric case, we have
{\small
\begin{equation}
\sum_{\ell=0}^d [\ell\ j\ k] = p^U_{jk'}, \qquad
\sum_{\ell=0}^d [i\ \ell\ k] = p^V_{ik'}, \qquad
\sum_{\ell=0}^d [i\ j\ \ell] = p^W_{ij'},
\label{eqn:triple}
\end{equation}
}
\begin{equation}
[0\ j\ k] = \delta_{j'W} \delta_{k'V}, \qquad
[i\ 0\ k] = \delta_{iW} \delta_{k'U}, \qquad
[i\ j\ 0] = \delta_{iV} \delta_{jU},
\label{eqn:triple2}
\end{equation}
where $(v, w) \in \Rrel_U$, $(u, w) \in \Rrel_V$, $(u, v) \in \Rrel_W$, and
primed indices correspond 
to the transposed relations.
\begin{proof}[Proof of \eqref{eqn:triple}, \eqref{eqn:triple2}]
Recall that $\Rrel_i(x)=\{y\in \Omega \mid (x,y)\in \Rrel_i \}$, and  
for $(x,y)\in \Rrel_k$, $p_{ij'}^k=|\Rrel_i(x)\cap \Rrel_j(y)|$. 
Then
{\allowdisplaybreaks
\begin{align*}
\sum_{\ell=0}^d [\ell\ j\ k] &=\sum_{\ell=0}^d|\Rrel_\ell(u)\cap \Rrel_j(v)\cap \Rrel_k(w)|\\
&=|\bigcup_{\ell=0}^{d}(\Rrel_\ell(u)\cap \Rrel_j(v)\cap \Rrel_k(w))|\\
&=|\Omega\cap \Rrel_j(v)\cap \Rrel_k(w)|\\
&=p_{jk'}^{U}.
\end{align*}}

The other equalities are obtained by similar calulations, completing the proof of \eqref{eqn:triple}. 

For \eqref{eqn:triple2}, 
by $(v, w) \in \Rrel_U$, $(u, w) \in \Rrel_V$, $(u, v) \in \Rrel_W$, one has 
\begin{align*}
[0\ j\ k]&=|\Rrel_0(u)\cap \Rrel_j(v) \cap \Rrel_k(w)| \\
&=|\{x\in \Omega\mid x=u, (v,x)\in \Rrel_j, (w,x)\in \Rrel_k\}|= \delta_{j'W} \delta_{k'V}
\end{align*}
The other equalities are obtained by similar calulations, completing the proof of \eqref{eqn:triple2}.
\qedhere
\end{proof}

\begin{prop}\label{prop:qijk=0}
Let $(\Omega, \{\Rrel_i\}_{i=0}^d)$ be a commutative association scheme of $d$ classes
with second eigenmatrix $\Qmat$
and Krein parameters $q_{ij}^k$ $(0 \le i,j,k \le d)$.
Then,
\[
q_{ij}^k = 0 \quad \Longleftrightarrow \quad
\sum_{r,s,t=0}^d \Qmat_{ri}\Qmat_{sj}\overline{\Qmat_{tk}}\sV{u & v & w}{r & s & t} = 0
\quad \mbox{for all\ } u, v, w \in \Omega.
\]
\end{prop}
\begin{proof}
Fix $i,j,k$.
Set $q(u,v,w)=\sum_{x\in \Omega}\Emat_i(u,x)\Emat_j(v,x)\overline{\Emat_k(w,x)}$ for $u,v,w\in \Omega$. 
Note that 
{\allowdisplaybreaks
\begin{align*}
q(u,v,w)&=\sum_{x\in \Omega}\Emat_i(u,x)\Emat_j(v,x)\overline{\Emat_k(w,x)}\\
&=\sum_{x\in \Omega}(\sum_{r=0}^d \frac{1}{|\Omega|}\Qmat_{ri}\Amat_r(u,x))(\sum_{s=0}^d \frac{1}{|\Omega|}\Qmat_{sj}\Amat_s(v,x))(\sum_{t=0}^d \frac{1}{|\Omega|}\overline{\Qmat_{tk}}\Amat_t(w,x))\\
&=\frac{1}{|\Omega|^3}\sum_{r,s,t=0}^d \Qmat_{ri}\Qmat_{sj}\overline{\Qmat_{tk}}(\sum_{x\in \Omega} \Amat_r(u,x)\Amat_s(v,x)\Amat_t(w,x))\\
&=\frac{1}{|\Omega|^3}\sum_{r,s,t=0}^d \Qmat_{ri}\Qmat_{sj}\overline{\Qmat_{tk}}\sV{u & v & w}{r & s & t}. 
\end{align*}}

Since the $\Emat_\ell$ are Hermitian idempotent matrices,  
\[
\Emat_\ell(x,y)=\sum_{z\in \Omega} \Emat_\ell(x,z)\Emat_\ell(z,y)=\sum_{z\in \Omega} \overline{\Emat_\ell(z,x)}\Emat_\ell(z,y) 
\]
for any $\ell$. 
Then by $\Emat_{k}^\top=\overline{\Emat_k}$, 
{\allowdisplaybreaks
\begin{align*}
&\sum_{x,y\in \Omega}(\Emat_i\circ \Emat_j \circ \Emat_{k}^\top)(x,y)\\
=&\sum_{x,y\in \Omega}(\Emat_i\circ \Emat_j \circ \overline{\Emat_{k}})(x,y)\\
=&\sum_{x,y\in \Omega}\Emat_i(x,y) \Emat_j(x,y) \overline{\Emat_k(x,y)}\\
=&\sum_{x,y\in \Omega}(\sum_{u\in \Omega} \overline{\Emat_i(u,x)}\Emat_i(u,y)) (\sum_{v\in \Omega} \overline{\Emat_j(v,x)}\Emat_j(v,y))(\sum_{w\in \Omega} \Emat_k(w,x)\overline{\Emat_k(w,y)})\\
=&\sum_{u,v,w\in \Omega}(\sum_{x\in \Omega} \overline{\Emat_i(u,x)}\overline{\Emat_j(v,x)}\Emat_k(w,x))(\sum_{y\in \Omega}\Emat_i(u,y)\Emat_j(v,y)\overline{\Emat_k(w,y)}))\\
=&\sum_{u,v,w\in \Omega}\overline{q(u,v,w)}q(u,v,w)\geq 0. 
\end{align*}}
On the other hand, since $\sum_{x,y\in \Omega} (\mathsf{A}\circ \mathsf{B})(x,y)=\text{tr}(\mathsf{AB}^\top)$, we have 
\begin{align*}
\sum_{x,y\in \Omega}(\Emat_i\circ \Emat_j \circ \Emat_{k}^\top)(x,y)&=\text{tr}((\Emat_i\circ \Emat_j)\Emat_k)=\text{tr}(\frac{1}{|\Omega|}\sum_{\ell=0}^dq_{ij}^{\ell} \Emat_{\ell}\Emat_k)\\
&=\text{tr}(\frac{1}{|\Omega|}q_{ij}^{k}\Emat_{k})=\frac{1}{|\Omega|}m_k q_{ij}^{k}.
\end{align*}
Therefore we have that $q_{ij}^{k}=0$ if and only if $q(u,v,w)=0$ for all $u,v,w\in \Omega$ as desired.  
\end{proof}

\section{Constructing roux matrices from association schemes}\label{sec:construction}

The following is a simple criteria for checking if a matrix is a roux matrix:

\begin{lem}[{\cite[Lemma 2.3]{IversonMixon}}]\label{lem:RouxSimple}
    Suppose $\Rmat \in \mathbb{C}[G]^{n \times n}$ satisfies \emph{(R1)}, \emph{(R2)}, and \emph{(R3)} of Definition \ref{def:rouxmat}. Then $\Rmat$ is a roux matrix if and only if
    \[
    \Rmat^2 = (n-1)\Imat + \sum_{g \in G} c_g g \Rmat
    \]
for some complex numbers $\{c_g\}_{g \in G}$, called the \emph{roux parameters}. In this case $\{c_g\}_{g \in G}$ are nonnegative integers that sum to $n-2$, with $c_{g^{-1}}=c_g$ for every $g \in G$.
\end{lem}

The key result of this section is the following,

\begin{thm}\label{lem:1130a}
Let $G$ be an abelian group of order $r$, and let $\cY$ be an $r$-class association scheme on $n-1\geqslant 2$ points
whose non-diagonal relations are $\Rrel_g$ ($g\in G$).
Define
\begin{equation}\label{1129a}
\tA_g=\begin{cases}
\begin{bmatrix}
0&\allone_{n-1}\\ \allone_{n-1}^\top&\Amat_1
\end{bmatrix}
&\text{if $g=1$,}\\
\begin{bmatrix}
0&\allzero_{n-1} \\ \allzero_{n-1}^\top&\Amat_g
\end{bmatrix}
&\text{if $g\in G\setminus\{1\}$,}
\end{cases}
\end{equation}
where $\Amat_g$ is the adjacency matrix of $\Rrel_g$ for $g\in G$,
and $\allone_{n-1}$ and $\allzero_{n-1}$ are row vectors with dimension $n-1$ with all $1$ and all $0$ entries, respectively.
The matrix 
\[
\Rmat = \sum_{g \in G} g \tA_g
\]
is a roux matrix if and only if both
\begin{equation}\label{1127a}
\Rrel_g^\top=\Rrel_{g^{-1}}
\end{equation}
and the intersection numbers 
of $\cY$
satisfy
\[
\sum_{g\in G} p_{g^{-1},gh}^m = k_{h^{-1}m} - \delta_{h,1}.
\]
for all $h, m \in G$. In such a case the roux parameters are the valencies $\{k_g\}_{g \in G}$ 
of $\cY$.
\end{thm}

\begin{proof}
It is clear that the multiplication $\tA_h \tA_k$ can be expressed using intersection numbers of the scheme $\cY$. 
Properties (R1) and (R2) clearly hold, and
(R3) holds if and only if $\Rrel_g^\top=\Rrel_{g^{-1}}$. We will apply Lemma \ref{lem:RouxSimple} to prove that (R4) holds if and only if the intersection numbers satisfy the equation in the theorem.
In the following, we omit the subscript from $\allone_{n-1}$ and $\allzero_{n-1}$ for clarity.
Observe that
\[
\Rmat = \sum_{g \in G} g \tA_g = \tA_1 + \sum_{g \neq 1} g \tA_g =
\begin{bmatrix}
0&\allone\\ \allone^\top&\Amat_1
\end{bmatrix}
+
\begin{bmatrix}
0& \allzero \\ \allzero^\top &\sum_{g \neq 1} g \Amat_g
\end{bmatrix}.
\]
Let $\Omega$ denote the point set of $\cY$. 
Computing $\Rmat^2$ we obtain
{\allowdisplaybreaks
\begin{align*}
\Rmat^2 = & \begin{bmatrix}
0&\allone\\ \allone^\top&\Amat_1
\end{bmatrix}^2
+
\begin{bmatrix}
0&\allone\\ \allone^\top&\Amat_1
\end{bmatrix}
\begin{bmatrix}
0& \allzero \\ \allzero^\top &\sum_{g \neq 1} g \Amat_g
\end{bmatrix} \\
& +
\begin{bmatrix}
0& \allzero \\ \allzero^\top &\sum_{h \neq 1} h \Amat_h
\end{bmatrix}
\begin{bmatrix}
0&\allone\\ \allone^\top&\Amat_1
\end{bmatrix}
+
\begin{bmatrix}
0& \allzero \\ \allzero^\top &\sum_{g \neq 1} g \Amat_g
\end{bmatrix}
\begin{bmatrix}
0& \allzero \\ \allzero^\top &\sum_{h \neq 1} h \Amat_h
\end{bmatrix}\\
=&
\begin{bmatrix}
|\Omega| &k_1 \allone\\ k_1\allone^\top&\Jmat + \Amat_1^2
\end{bmatrix}
+
\begin{bmatrix}
0&\sum_{g\neq 1}k_g g \allone \\ \allzero^\top& \sum_{g\neq1}g\Amat_1\Amat_g
\end{bmatrix} \\
& +
\begin{bmatrix}
0&\allzero \\ \sum_{h\neq1} k_h h\allone^\top & \sum_{h\neq1} h \Amat_h \Amat_1
\end{bmatrix}
 +
 \begin{bmatrix}
0&\allzero \\ \allzero^\top &\sum_{g\neq1, h\neq1} gh \Amat_g\Amat_h
\end{bmatrix}\\
= & 
\begin{bmatrix}
|\Omega| &\sum_{g\in G} k_g g \allone\\ \sum_{g \in G} k_g g \allone^\top&\Jmat + \sum_{g,h \in G}gh \Amat_g \Amat_h
\end{bmatrix}.
\end{align*}}
Note
{\allowdisplaybreaks
\begin{align*}
\sum_{g,h \in G} gh\Amat_g \Amat_h = & \sum_{g,h \in G, \ell \in G \cup \{0\}} p^\ell_{g,h} gh \Amat_\ell = \sum_{g,h \in G} p^0_{g,h} gh \Amat_0 + \sum_{g,h,\ell \in G} p^\ell_{g,h} gh \Amat_\ell \\
=& \sum_{g \in G} k_g \Imat + \sum_{\ell \in G} \sum_{g,h \in G} p^\ell_{g,h} gh \Amat_\ell = (|\Omega| - 1)\Imat + \sum_{\ell \in G} \sum_{g,h \in G} p^\ell_{g,h} gh \Amat_\ell.
\end{align*}}
Thus by writing $n = |\Omega|+1$ we have
\begin{equation}\label{eq:Bsquared}
\Rmat^2 = (n-1)\Imat +
\begin{bmatrix}
0 &\sum_{g\in G} k_g g \allone\\ \sum_{g \in G} k_g g \allone^\top&\Jmat - \Imat + \sum_{\ell \in G} \sum_{g,h\in G} p^\ell_{g,h} gh \Amat_\ell
\end{bmatrix}.
\end{equation}
Simplifying the bottom right corner of \eqref{eq:Bsquared} gives
{\allowdisplaybreaks
\begin{align}
\Jmat-\Imat + \sum_{\ell \in G} \sum_{g,h \in G} p^\ell_{g,h} gh \Amat_\ell \nonumber
=& \Jmat-\Imat + \sum_{\ell \in G} \sum_{g, h \in G} p^\ell_{g,h\ell} gh\ell \Amat_\ell\\ \nonumber
=&\Jmat-\Imat + \sum_{\ell \in G} \sum_{g', h \in G} p_{g'h^{-1}, h\ell}^\ell g'\ell \Amat_\ell\\ \nonumber
=& \sum_{\ell,g \in G} \delta_{g\ell, 1} g\ell \Amat_\ell + \sum_{\ell, g \in G} \sum_{h \in G} p^\ell_{gh^{-1}, h\ell} g\ell \Amat_\ell\\
=& \sum_{\ell, g \in G} (\delta_{g\ell, 1} + \sum_{h \in G} p_{gh^{-1}, h\ell}^\ell) g\ell \Amat_\ell. \label{eq:simplified}
\end{align}}

Since
\[
g\Rmat = 
\begin{bmatrix}
0&g \allone \\ g \allone^\top & g \Amat_1
\end{bmatrix}
+
\begin{bmatrix}
0& \allzero \\ \allzero^\top & \sum_{\ell \neq 1} g \ell \Amat_\ell
\end{bmatrix}
=
\begin{bmatrix}
0&g \allone \\ g \allone^\top & 0
\end{bmatrix}
+
\begin{bmatrix}
0& \allzero \\ \allzero^\top & \sum_{\ell \in G} g \ell \Amat_\ell
\end{bmatrix}
\]
it follows that $c_g = k_g$ for all $g \in G$ if \eqref{eq:Bsquared} is to satisfy the equation in Lemma \ref{lem:RouxSimple}, and so we have
\begin{equation} \label{eq:combo}
\sum_{g \in G} k_g g \Rmat = 
\begin{bmatrix}
0& \sum_{g \in G} k_g g \allone \\ \sum_{g \in G} k_g g \allone^\top & 0
\end{bmatrix}
+
\begin{bmatrix}
0& \allzero \\ \allzero^\top & \sum_{\ell, g \in G} k_g g \ell \Amat_\ell
\end{bmatrix}.
\end{equation}

Comparing and equating \eqref{eq:simplified} and \eqref{eq:combo}
\[
\Jmat-\Imat + \sum_{\ell \in G} \sum_{g, h \in G} p_{gh}^\ell gh \Amat_\ell = \sum_{\ell, g \in G} k_g g \ell \Amat_\ell
\]
we obtain $\delta_{g\ell, 1} + \sum_{h \in G} p_{g h^{-1}, h\ell}^\ell = k_g$ for all $g, \ell \in G$. Equivalently $\delta_{g\ell, 1} + \sum_{g' \in G} p_{g'^{-1}, g'g\ell}^\ell = k_g$. Equivalently $\sum_{g \in G}p_{g^{-1}, gh}^m = k_{hm^{-1}} - \delta_{h,1}$ for all $h, m \in G$.
\end{proof}

Given a group $G$ and commutative association scheme $\cY$ which produce a roux matrix of the form in Theorem \ref{lem:1130a}, we shall denote the corresponding roux scheme by $\cY \widehat\otimes G$. The reason for this notation is the following observation.
For $k\in G$, let $\Perm_k$ be the permutation matrix whose rows and columns are indexed
by $G$, whose entries are defined by
\[(\Perm_k)_{g,h}=\delta_{gk,h}\quad(h,k\in G).\]
Then
\begin{equation}\label{1128a2}
\Perm_k\Perm_\ell=\Perm_{k\ell}\quad(k,\ell\in G).
\end{equation}

Then for $k\in G$, the adjacency matrices of thick and thin relations of $\cY \widehat\otimes G$ respectively have the structure
\begin{align}
\Bmat_k&=\sum_{\ell\in G} \tA_\ell\otimes \Perm_{k\ell},\label{thickrels}\\
\Cmat_k&=\Imat\otimes \Perm_k. \label{thinrels}
\end{align}

\begin{remark} 
Note that a labeling of the relations of $\cY$ by the elements of $G$ is implicit in the notation $\cY \widehat\otimes G$.
\end{remark}

It is clear that $\Smat_k$ has valency equal to $n-1$, the number of points of $\cY$. Moreover, the roux scheme has $2r-1$ classes and $rn$ points with $r < n-1$ since $\cY$ has $r$ classes. Given that $\Smat_k = \Smat_1 \Tmat_k$ it follows that $\Smat_k\Smat_g = \Smat_1^2 \Tmat_k \Tmat_g$ and so the intersection numbers of the scheme are determined by $\Smat_1^2$ expressed in terms of the adjacency matrices.
One can see that the roux scheme $\cY \widehat\otimes G$ carries an association scheme locally 
with respect to any thick relation 
about some vertex\footnote{In particular, the vertex whose first tensor product component corresponds to the first index in the matrix $\tilde{A}_g$ in \eqref{1129a}.}.
Conversely, 
in Section \ref{sec:Decomp}, 
we show that a roux scheme with 
the latter property can be obtained as $\cY \widehat\otimes G$ (the roux scheme from the roux matrix in Theorem \ref{lem:1130a}).

\subsection{Constructions using elementary abelian groups}

A symmetric $d$-class association scheme with the property that every non-identity relation has the same valency $k$, and $\sum_{i}p_{ii}^j = k-1$ for $1 \leqslant j \leqslant d$, is called \emph{pseudocyclic} \cite[Proposition 2.2.7]{BCN}. A symmetric association scheme with the property that any of its fusions is also an association
scheme is called \emph{amorphic}.

\begin{lem}\label{lem:pseudocyclicamorphic}
Let $\cY$ be an amorphic pseudocyclic association scheme with $r$ classes and let $G$ be an elementary abelian $2$-group of order $r$.
Then $\cY$ and $G$ satisfy the conditions of Theorem \ref{lem:1130a} to form a roux matrix.
\end{lem}
\begin{proof}
For an amorphic pseudocyclic association scheme \cite[Corollary 1]{vanDamMuzychuk}, 
one has
\begin{align*}
k &= g(rg \mp 2), & p_{ii}^i &= g^2 - 1 \pm g(r-3),& p_{ii}^j &= g(g\mp 1), & p_{ij}^k &= g^2,
\end{align*}
for $1 \leqslant i, j, k \leqslant r$ and $i \neq j$, $i \neq k$, $j \neq k$, where the upper sign corresponds to Latin square type and the lower sign corresponds to negative Latin square type. Moreover, $g = g^{-1}$ for all $g\in G$ since $G$ is an elementary abelian $2$-group. The left hand side of the equation in Theorem \ref{lem:1130a} then evaluates to
\[
(r-2)g^2 + 2g(g\mp 1) = g^2r\mp 2g
\]
for $h \neq 1$ and 
\[
(r-1)g(g\mp1) + g^2 -1 \pm g (r-3) = g^2r \mp 2g - 1
\]
for $h = 1$, matching the expression on the right hand side of the equation. Note that this holds for any ordering of elements of $G$.
\end{proof}

Note, for any $s \in \mathbb{Z}_+$ and $r=2^s$, there are infinitely many primes $p$ satisfying $p \equiv -1 \pmod{r}$ (by Dirichlet's theorem on arithmetic progressions) for which there exists a cyclotomic scheme over $\mathbb{F}_{p^n}$ with $r$ classes, where $n$ is an even positive integer. This scheme is amorphic by \cite[Proposition 3]{vanDamMuzychuk} and hence there are infinitely many examples arising from Lemma \ref{lem:pseudocyclicamorphic}.

We have the following corollary as a result of \cite[Theorem 4.2(d)]{IversonMixon} (which is a special case of \cite{Mathon}, c.f. \cite[Corollary 12.7.2]{BCN}),
\begin{cor}
Let $r$ be a power of $2$, and suppose that there exists an amorphic pseudocyclic association scheme with $r$ classes on $n$ points. Then there exists a DRACKN with parameters $(n+1,r,\frac{n-1}{r})$.
\end{cor}

\subsection{Constructions using cyclic groups}

We are aware of the following three examples which use cyclic groups to construct roux schemes. We assume that the relations of $\cY$ are labeled by $R_i = R_{g^i}$ when we refer to $\cY \widehat{\otimes} G$.

\begin{exmp}
Let $\cY$ be the unique association scheme on $8$ vertices with the eigenmatrix
\[\begin{bmatrix}
1 & 1 & 3 & 3\\
1 & -1 & \sqrt{3}i & -\sqrt{3}i\\ 
1 & -1 & -\sqrt{3}i & \sqrt{3}i\\
1 & 1 & -1 & -1 
\end{bmatrix}.\]
In Hanaki's database \cite{Hanaki}, $\cY$ has identification number $6$. It is Schurian with automorphism group isomorphic to ${\rm SL}(2,3)$. Then $\cY$ and $\mathbb{Z}_3$ together satisfy Theorem \ref{lem:1130a}. Moreover, $\cY \widehat\otimes \mathbb{Z}_3$ is the unique association scheme on $27$ vertices with eigenmatrix
\[\begin{bmatrix}
1 & 1 & 1 & 8 & 8 & 8\\
1 & \zeta & \zeta^2 & -4 & -4\zeta & -4\zeta^2\\
1 & \zeta^2 & \zeta & -4 & -4\zeta^2 & -4\zeta \\ 
1 & \zeta & \zeta^2 & 2 & 2\zeta & 2\zeta^2\\
1 & \zeta^2 & \zeta & 2 & 2\zeta^2 & 2\zeta\\
1 & 1 & 1 & -1 & -1 & -1
\end{bmatrix},\]
where $\zeta$ is a third root of unity, and has identification number $403$ in Hanaki's database \cite{Hanaki}. It is the association scheme that arises from the set of $9$ equiangular lines in $\mathbb{C}^3$ in the same way as Hoggar's lines produces the association scheme $\mathcal{H}$ in Section \ref{sec:intro}.
\end{exmp}

\begin{exmp}
Let $\mathcal{O}$ be an ovoid of 
the Hermitian polar space $H(3,4)$ constructed by taking the intersection with non-degenerate hyperplane of $PG(3,4)$. The lines of $H(3,4)$ form a $3$-class association scheme with the following relations:
\begin{itemize}
    \item $R_1$: Two distinct lines meet in a point not in $\mathcal{O}$.
    \item $R_2$: Two distinct lines are disjoint.
    \item $R_3$: Two distinct lines meet in a point of $\mathcal{O}$. 
\end{itemize}
This scheme is metric, and $R_1$ defines a distance-regular graph. There is a fission of this association scheme into a $4$-class scheme by splitting $R_2$ in half, which has the eigenmatrix
\[\begin{bmatrix}
1&  2&         8&   8&        8\\
1&  2&        -1&  -1&       -1\\
1& -1&         2&  -4&        2\\
1& -1& -1-3i    &   2&      -1+3i\\
1& -1& -1+3i    &   2&      -1-3i
\end{bmatrix}.\]
Let $\cY$ be this fission scheme. It is the unique association scheme on $27$ vertices with identification number $395$ in Hanaki's database \cite{Hanaki}. It is also Schurian.
Then $\cY$ and $\mathbb{Z}_4$ together satisfy Theorem \ref{lem:1130a}.
\end{exmp}

\begin{exmp} \label{ex:hex}
The collinearity graph of the split Cayley generalised hexagon, $H(2)$, is diameter $3$ distance-regular graph. Equivalently, it is a $3$-class metric association scheme. There exists a $4$-class fission of this scheme, by splitting the distance $3$ relation. It is Schurian with automorphism group isomorphic to ${\rm PSU}(3,3)$ and with eigenmatrix
\begin{align}\label{eq:HexFissionP}
\begin{bmatrix}
1&6&16&24&16\\
1&3&-2&0&-2\\
1&-1&2&-4&2\\
1&-3&-2+6i&6&-2-6i\\
1&-3&-2-6i&6&-2+6i
\end{bmatrix}.
\end{align}
Note that ${\rm PSU}(3,3)$ has two subgroups of index $63$ up to conjugacy. One of them \cite[p.14]{ATLAS} is the stabiliser of a pairwise orthogonal set of three non-isotropic points, i.e. a ``basis''. 
The orbitals of the action of ${\rm PSU}(3,3)$ on the $63$ bases produces this $4$-class fission scheme, which we denote by $\cY$. Then $\cY$ and $\mathbb{Z}_4$ together satisfy Theorem \ref{lem:1130a}. Moreover, $\cY \widehat\otimes \mathbb{Z}_4$ is the association scheme $\mathcal{H}$ on $256$ vertices arising from Hoggar's lines (see Section \ref{sec:intro}) with eigenmatrix
\begin{align}\label{eq:2}
\left[
\begin{array}{cccccccc}
1 & 1 & 1 & 1 & 63 & 63 & 63 & 63\\
1 & i & -1 & -i & 21 & 21i & -21 & -21i\\
1 & -1 & 1 & -1 & -9 & 9 & -9 & 9\\
1 & -i & -1 & i & 21 & -21i & -21 & 21i\\
1 & 1 & 1 & 1 & -1 & -1 & -1 & -1\\
1 & i & -1 & -i & -3 & -3i & 3 & 3i\\
1 & -1 & 1 & -1 & 7 & -7 & 7 & -7\\
1 & -i & -1 & i & -3 & 3i & 3 & -3i
\end{array}
\right].
\end{align}
We shall see in Theorem \ref{thm:1} that it is the unique scheme with this eigenmatrix.
\end{exmp}

\begin{remark}
    Note that the collinearity graph of the dual split Cayley hexagon is also distance-regular with the same parameters as that in Example \ref{ex:hex}, however it does not  admit a $4$-class fission scheme with this eigenmatrix. We will see in Theorem \ref{thm:2} that the fission scheme $\cY$ in Example \ref{ex:hex} is unique up to isomorphism.
\end{remark}

\section{Characterising roux schemes by their eigenmatrices}\label{sec:EigenCharacterisation}
In this section we provide the eigenmatrices for a roux scheme. We note that the primitive idempotents of general roux schemes are given (after suitable scaling)  in \cite[Theorem 2.8]{IversonMixon} and so the information in this section is essentially known. However it is useful to have explicitly stated in the form of the  eigenmatrices of the association scheme. Moreover, we then show that this eigenmatrix structure characterises roux schemes.

\begin{thm}[Eigenmatrices]\label{thm:Pmat}
Let $\Rmat\in \mathbb{C}[G]^{n \times n}$ 
be a roux matrix for a
group $G$.
The first and second eigenmatrices of a roux scheme are respectively
\begin{equation}\label{eq:mat}
\left[\begin{array}{@{}c|c@{}}
  \begin{matrix}
   & \vdots & \\
 \hdots  & \overline{\alpha(g)} &\hdots  \\
   & \vdots & \\
  \end{matrix}
  &  
    \begin{matrix}
   & \vdots & \\
 \hdots  & (n-1)\mu_\alpha^+ \overline{\alpha(g)} & \hdots \\
   & \vdots & \\
  \end{matrix}
   \\
\hline
  \begin{matrix}
   & \vdots & \\
 \hdots  & \overline{\alpha(g)} & \hdots \\
   & \vdots & \\
  \end{matrix}
 &
  \begin{matrix}
   & \vdots & \\
  \hdots & (n-1)\mu_\alpha^- \overline{\alpha(g)} & \hdots \\
   & \vdots & \\
  \end{matrix}
\end{array}\right],\quad
\left[\begin{array}{@{}c|c@{}}
  \begin{matrix}
   & \vdots & \\
 \hdots  & \lambda_\alpha^+\alpha(g) &\hdots  \\
   & \vdots & \\
  \end{matrix}
  &  
    \begin{matrix}
   & \vdots & \\
 \hdots  & \lambda_\alpha^- \alpha(g) & \hdots \\
   & \vdots & \\
  \end{matrix}
   \\
\hline
  \begin{matrix}
   & \vdots & \\
 \hdots  & \lambda_\alpha^+ \mu_\alpha^+ \alpha(g) & \hdots \\
   & \vdots & \\
  \end{matrix}
 &
  \begin{matrix}
   & \vdots & \\
  \hdots & \lambda_\alpha^- \mu_\alpha^- \alpha(g) & \hdots \\
   & \vdots & \\
  \end{matrix}
\end{array}\right],
\end{equation}
where $g \in G$ determine columns of the first eigenmatrix and rows of the second, $\alpha \in \widehat{G}$ (the character group of $G$) determine rows of the first eigenmatrix and columns of the second, and
\[
\widehat{c}_\alpha := \sum_{h \in G} c_h  
\overline{\alpha(h)}, \quad
\mu_\alpha^\epsilon := \frac{\widehat{c}_\alpha + \epsilon \sqrt{(\widehat{c}_\alpha)^2 + 4 (n-1)}}{2(n-1)},\quad
\lambda_\alpha^\epsilon := \frac{n}{2+ \widehat{c}_\alpha \mu_\alpha^\epsilon}
\]
for $\epsilon \in \{+, -\}$ and roux parameters $\{c_g\}_{g \in G}$.
\end{thm}

\begin{proof}
    Up to a scalar multiple, the primitive idempotents of a roux scheme are provided in \cite[Theorem 2.8]{IversonMixon}, and the exact value of the scalar is found in the body of the proof, as is the equality $1 + (n-1) (\mu_\alpha^\epsilon)^2 = 2 + \widehat{c}_\alpha \mu_\alpha^\epsilon$. 
    Let $r$ be the order of $G$. Then the primitive idempotents are
\[
\frac{1}{r(1 + (n-1)(\mu_\alpha^\epsilon)^2)} \mathsf{G}_\alpha^\epsilon
\]
for
    \[
    \mathsf{G}_\alpha^\epsilon = \sum_{g \in G} \alpha(g)\lceil g \Imat \rfloor + \mu_\alpha^\epsilon \sum_{g \in G} \alpha(g) \lceil g\Rmat \rfloor.
    \]
    The coefficients in this expression give the columns of the second eigenmatrix, where the first sum corresponds to the thin relations and the second sum corresponds to the thick relations.
Observe that
\[
        \frac{rn}{r(1+(n-1)(\mu_\alpha^\epsilon)^2)}
        =\frac{n}{2 + \widehat{c}_\alpha \mu_\alpha^\epsilon},
\]
and setting this value to $\lambda_\alpha^\epsilon$ gives the entries in the second eigenmatrix.
The multiplicities (called $d_\alpha^\epsilon$ in \cite{IversonMixon}) are given by $m_\alpha^\epsilon = \frac{rn}{r(1 + (n-1)(\mu_\alpha^\epsilon)^2)} = \frac{n}{2 + \widehat{c}_\alpha \mu_\alpha^\epsilon}$. Observe that $\lambda_\alpha^\epsilon/ m_\alpha^\epsilon = \frac{n}{2+ \widehat{c}_\alpha \mu_\alpha^\epsilon}\frac{2 + \widehat{c}_\alpha \mu_\alpha^\epsilon}{n} = 1$. 
Denote the second eigenmatrix by $\Qmat$ and the first eigenmatrix by $\Pmat$, then $\Pmat = \Delta_m^{-1} \overline{\Qmat^\top} \Delta_k$, where $\Delta_k$ and $\Delta_m$ are matrices with the valencies and multiplicities on the diagonal, respectively. We know that the valencies of the thin relations are $1$ and the thick relations all have valency $n-1$, and so we obtain the first eigenmatrix.
\end{proof}

In response to the remark ``it seems that $d_\alpha^\epsilon \in \mathbb{Z}$ is a strong necessary condition for the existence
of roux" \cite[p22]{IversonMixon} we note that, since these values are the multiplicities of the association scheme, they must be positive integers.

If the association scheme is pseudocyclic (such as in Lemma \ref{lem:pseudocyclicamorphic}) the eigenmatrix structure is a natural consequence of constant valencies:

\begin{cor}
Let $\cY$ be a pseudocyclic association scheme, then $\cY \widehat\otimes G$ has eigenmatrix as given in Theorem \ref{thm:Pmat}, where 
\[
\mu_{1_G}^+ = 1, \mu_{1_G}^-=-\frac{1}{n}, 
\text{ and } \mu_\alpha^\epsilon= \frac{\epsilon}{\sqrt{n}}, 
\]
for ${1_G}$ the identity character and ${1_G} \neq \alpha \in \widehat{G}$.
\end{cor}

\begin{proof}
Since $\cY$ is pseudocyclic, $k_g$ is constant for all $g \in G$.
So $\widehat{c}_{1_G} = n-1$ and $(n-1)^2 + 4n = (n+1)^2$. If  $\alpha \neq 1_G$ then $\widehat{c}_\alpha = 0$. The rest follows directly.
\end{proof}

We observe that the block structure of the eigenmatrix consists of the character table of $G$, with scaled rows. We are able to characterise roux schemes by their eigenmatrices as a result of Theorem \ref{thm:Pmat} together with Theorem \ref{thm:thinreg_basic} below.

\begin{thm}\label{thm:thinreg_basic}
Let $\cX$ be a commutative association scheme such that:
\begin{enumerate}
\item The eigenmatix of $\cX$ has the form
\begin{equation}\label{eq:PMatFormReg}
\Pmat=  \begin{bmatrix}
    \mathbb{T} & \mathsf{D}_+\mathbb{T} \\
    \mathbb{T} & \mathsf{D}_-\mathbb{T}
  \end{bmatrix},
\end{equation}
where $\mathsf{D}_+$ and $\mathsf{D}_-$ are some diagonal matrices and $\mathbb{T}$ is the character table of the thin radical of $\cX$,
\item  There is at least one entirely real column of $\Pmat$ corresponding to a thick relation.
\end{enumerate}
Then $\cX$ is a roux scheme.
\end{thm}

\begin{proof}
Let $G$ be the thin radical and its order be $r$.
For a
thin relation $g \in G$ we denote the corresponding adjacency matrix by $\Tmat_g$.
It is clear from the block structure of $\Pmat$ that there are $2r$ relations in total.
Hence we may also label the thick relations by elements of $G$ and denote their adjacency matrices by $\Smat_g$.
Since there are $2r$ relations there are also $2r$ primitive idempotents of the Bose-Mesner algebra, and since $\mathbb{T}$ repeats twice we can connect two primitive idempotents to each character in $\widehat{G}$. Thus, we shall denote the primitive idempotents by $\Emat_\alpha^\epsilon$ for $\alpha \in \widehat{G}$ and $\epsilon \in \{+,-\}$.
The relations can be expressed as linear combinations of the primitive idempotents by the columns of $\Pmat$ like so:
\[
\Tmat_g = \sum_{\alpha \in \widehat{G}} (\mathbb{T}_{\alpha, g} \Emat_\alpha^+ + \mathbb{T}_{\alpha, g} \Emat_{\alpha}^-),\quad 
\Smat_g = \sum_{\alpha \in \widehat{G}} (d_{\alpha}^+\mathbb{T}_{\alpha, g} \Emat_\alpha^+ + d_{\alpha}^-\mathbb{T}_{\alpha, g} \Emat_{\alpha}^-),
\]
where $d_\alpha^\epsilon$ is the diagonal entry of $\mathsf{D}_\epsilon$ on the $\alpha$ row.
Since $\Emat_\alpha^\epsilon$ are idempotent for all $\alpha \in \widehat{G}$ and $\epsilon \in \{+,-\}$ and $\alpha(h)\alpha(g) = \alpha(hg)$ it follows that for all $h, g \in G$, 
\[
\Tmat_h \Smat_g  = \sum_{\alpha \in \widehat{G}} (d_\alpha^+ \mathbb{T}_{\alpha, h} \mathbb{T}_{\alpha, g} \Emat_\alpha^+ + d_\alpha^- \mathbb{T}_{\alpha, h}  \mathbb{T}_{\alpha, g }\Emat_{\alpha}^-)
= \sum_{\alpha \in \widehat{G}} (d_\alpha^+ \mathbb{T}_{\alpha, hg} \Emat_\alpha^+ + d_\alpha^- \mathbb{T}_{\alpha, hg} \Emat_\alpha^-)
= \Smat_{hg}.
\]
(Moreover since $\cX$ is commutative the order of multiplication does not matter). Thus the thin relations act regularly on the thick relations, and since a column of the eigenmatrix is real if and only if the corresponding relation is symmetric, $\cX$ possesses a thick symmetric relation. 
It follows from Lemma \ref{lem:RouxRadicalCharacterisation} that $\cX$ is a roux scheme.
\end{proof}

Theorem \ref{thm:thinreg_basic} is useful for identifying roux schemes from their parameters alone, which we will do in Section \ref{sec:Uniqueness}.

\section{Decomposing roux schemes}\label{sec:Decomp}

Recall that roux schemes are characterised (Lemma \ref{lem:RouxRadicalCharacterisation}) as association schemes that possess a thick symmetric relation and whose thin radicals act regularly on the thick relations.
We consider the slightly more general case (with no symmetric relation) in the theorem below to show how an association scheme may be decomposed into blocks by the action of the thin radical. Later we will consider how roux schemes in particular can be decomposed when they carry an association scheme locally.

\begin{thm}\label{thm:decomp}
Let $\cX$ be a commutative association scheme such that the thin radical $G$ acts regularly on the thick relations. Then $\cX$ is uniquely recoverable from $G$ and the neighourhood stucture of any thick relation about any vertex.
(That is, the restriction of the relations of $\cX$ to this neighbourhood, together with the natural permutation representation of $G$, uniquely determine all relations of $\cX$.)
\end{thm}
\begin{proof}
We may label the thick relations by the elements of $G$. Let $\Srel_1$ be the thick relation mentioned in the theorem. Then the labeling of the remaining thick relations $S_g$ is determined by the regular action of $G$ such that $\Smat_g = \Smat_1\Tmat_g$ (where, as in the proof of Theorem \ref{thm:thinreg_basic}, we denote by $\Tmat_g$ the adjacency matrix of the thin relation $g \in G$ and by $\Smat_g$ the adjacency matrix of the thick relation indexed by $g$).

Let $\Omega$ be the vertex set of $\cX$ and define the right action of $G$ on $\Omega$ 
by $\{x^g\} = \Trel_g(x)$ for all $x \in \Omega$ and all $g \in G$. Equivalently, $y = x^g$ if and only if $(\Cmat_g)_{xy}=1$, and moreover since the relation $T_g$ is thin, it follows that $(\Cmat_g)_{x,y} = \delta_{x^g, y}$ for $x, y \in \Omega$. This is indeed an action, since
\begin{align*}
 (\Cmat_{gh})_{xz} =  (\Cmat_g \Cmat_h)_{xz} = (\Cmat_g)_{x x^g} (\Cmat_h)_{x^gz} = 1
  \iff   (\Cmat_h)_{x^gz}=1,
\end{align*}
hence $x^{gh} = (x^g)^h$.
Moreover, this action is clearly faithful and semi-regular.

Fix $z \in \Omega$ and set 
$\widetilde{\Omega}_1 := \{z\} \cup \Srel_1(z)$. 
Now
\[
\Srel_1(z)^g = \{y^g \mid y \in \Srel_1(z)\} = \bigcup_{y \in \Srel_1(z)} \Trel_g(y) = \Srel_g(z)
\]
and so
\[
\widetilde{\Omega}_g := \widetilde{\Omega}_1^g = \{z^g\} \cup \Srel_1(z)^g = \Trel_g(z) \cup \Srel_g(z).
\]
Moreover, $(\widetilde{\Omega}_g)^h = \widetilde{\Omega}_{gh}$ and
\[
\Omega = \bigcup_{g \in G} \left(\Trel_g(z) \cup \Srel_g(z) \right) = \bigcup_{g\in G} \widetilde{\Omega}_g.
\]

To determine the structure of $\Cmat_g$ and $\Bmat_g$ we will use the bijection $\widetilde{\Omega}_1 \times G \to \Omega$ given by $(x, g) \mapsto x^g$.
Firstly, we claim that $\Cmat_g = \Imat \otimes \Perm_g \in \{0,1\}^{\Omega \times \Omega}$.
Indeed, for $x, y \in \widetilde{\Omega}_1$, $h, k \in G$,
{\allowdisplaybreaks
\begin{align*}
(\Imat \otimes \Perm_g)_{(x,h), (y, k)} 
 &= \delta_{x,y}\delta_{hg,k}\\
 &= \delta_{(x,hg), (y,k)}\\
 &= \delta_{x^{hg}, y^k}\\
&= (\Cmat_g)_{x^h, y^k},
\end{align*}}
proving the claim.

Next, given $\ell\in G$, we define $\Mmat_\ell \in \{0,1\}^{\widetilde{\Omega}_1 \times \widetilde{\Omega}_1}$ 
as the adjacency matrix of the restriction of $S_{\ell^{-1}}$ to $\widetilde{\Omega}_1$, 
$\Mmat_\ell = \Bmat_{\ell^{-1}}|_{\widetilde{\Omega}_1 \times \widetilde{\Omega}_1}$. 
Note that $\{\Mmat_\ell\}_{\ell \in G}$ are precisely the matrices describing the relations
between vertices in the 
$S_1$-neighbourhood about $z$. 

We claim that $\Bmat_g = \sum_{\ell} \Mmat_\ell \otimes \Perm_{g\ell}$. 
Indeed, observing that $\Smat_{k^{-1}hg} = \Smat_1 \Tmat_{k^{-1}hg} = \Smat_1 \Tmat_{k^{-1}} \Tmat_h \Tmat_g = \Tmat_h \Smat_1 \Tmat_g \Tmat_{k^{-1}} = \Tmat_h \Smat_g \Tmat_{k^{-1}}$ since $\cX$ is commutative,
 for all $x, y \in \widetilde{\Omega}_1$ and all $h, k \in G$, we have
{\allowdisplaybreaks
\begin{align*}
\left(\sum_{\ell \in G} \Mmat_\ell \otimes \Perm_{g\ell}\right)_{(x,h), (y, k)} &= \sum_{\ell \in G} (\Mmat_\ell)_{x,y} (\Perm_{g\ell})_{h,k}\\
 &= \sum_{\ell \in G}(\Bmat_{\ell^{-1}})_{x,y}\delta_{hg\ell,k}\\
 &= (\Bmat_{k^{-1}hg})_{x,y}\\
 &=(\Cmat_h\Bmat_g\Cmat_{k^{-1}})_{x,y}\\
 &=\sum_{u,v}  (\Cmat_h)_{x,u} (\Bmat_g)_{u,v} (\Cmat_{k^{-1}})_{v,y}\\
 &= \sum_{u,v} \delta_{x^h, u} (\Bmat_g)_{u,v} \delta_{v^{k^{-1}},y}\\
 &= (\Bmat_g)_{x^h, y^k}
\end{align*}}
proving the claim. 

Now, the structure of the adjacency matrices
$\Cmat_g = \Imat \otimes \Perm_g$ and $\Bmat_g = \sum_{\ell} \Mmat_\ell \otimes \Perm_{g\ell}$ makes it clear how $\cX$ may be recovered from $G$ and $\{\Mmat_\ell\}_{ \ell \in G}$.
\end{proof}

\begin{lem}\label{lem:roux}
If $(\Omega,\{T_g\mid g\in G\}\cup\{S_g\mid g\in G\})$ 
is a roux scheme,
then there exists a labeling of thick relations in such a way that
\[
\Tmat_g\Smat_h=\Smat_{gh} \quad(g,h\in G), \quad \text{ and } \quad
S_g^\top=S_{g^{-1}}\quad(g\in G)
\]
hold.
\end{lem}
\begin{proof}
Since the scheme is roux, there is a thick 
symmetric
relation, which we label $S_1$.  
Since $\{\Tmat_g\}$ act regularly on the thick relations, this determines the remaining labeling of thick relations such that $\Smat_g = \Tmat_g \Smat_1$.
Then
$\Smat_g^\top = (\Tmat_g \Smat_1)^\top = \Tmat_{g^{-1}}\Smat_1 = \Smat_{g^{-1}}$, where the second equality comes from $S_1^\top = S_1$.
\end{proof}

\begin{thm}\label{thm:Rouxdecomp}
Let $\cX$ be a roux scheme with thin radical $G$. If the neighbourhood of some vertex with respect to some thick relation induces a $|G|$-class local association scheme, $\cY$, then $\cX = \cY \widehat\otimes G$.
\end{thm}

\begin{proof}
By Lemma \ref{lem:roux}, we may assign an ordering to the thick relations by the elements of $G$ such that $S_g^\top = S_{g^{-1}}$. In particular, we have $S_1$ is symmetric.

Let $\Omega$ be the vertex set of $\cX$. Recall the right action of $G$ on $\Omega$ defined 
in the proof of Theorem \ref{thm:decomp}.
Note that for all $x \in \Omega$,  $S_g(x^h) = S_{gh}(x)$ for all $h, g \in G$ and $T_g \cap (S_k(x) \times S_k(x))= \varnothing$ for all $k,g \in G$ except for $g=1$.
By assumption, there is some vertex $x$ and some thick relation $S_m$, $m \in G$, which induce the local scheme $\cY$ with vertex set $\Omega_1 = S_m(x)$, and hence $\{S_{\ell}|_{\Omega_1\times \Omega_1}\colon \ell \in G\} \cup \{T_{1}|_{\Omega_1 \times \Omega_1}\}$ are the relations of $\cY$. 
 
Let $z = x^m$, then $\Omega_1 = 
S_m(x) = S_1(x^m) = S_1(z)$ 
and so restricting to the $S_1$ neighbourhood about $z$ also yields $\cY$. From Theorem \ref{thm:decomp} (and its proof) we know that $\Cmat_g = \Imat \otimes \Perm_g$ and $\Bmat_g = \sum_{\ell} \tA_\ell \otimes \Perm_{g\ell}$ for $\tA_\ell = \Bmat_{\ell^{-1}}|_{\widetilde{\Omega}_1 \times \widetilde{\Omega}_1}$, and note that $\tA_\ell$ has the form \eqref{1129a} 
for $\Amat_\ell = \Bmat_{\ell^{-1}}|_{\Omega_1 \times \Omega_1}$. Finally $S_g^\top = S_{g^{-1}}$ implies that $\Amat_g^\top = \Amat_{g^{-1}}$ for all $g \in G$, and since the relations of $\cX$ are of the form \eqref{thickrels} and \eqref{thinrels}, it follows that $\cX = \cY \widehat\otimes G$.
\end{proof}

\begin{remark}
It is natural to consider generalisations such as when $\cX$ in Theorem \ref{thm:Rouxdecomp} is not roux (has no thick symmetric relations) or the number of classes of $\cY$ is less than $|G|$. However since this does not fit the theme of this paper, these generalisations will be treated separately in a forthcoming work by the second and third author.
\end{remark}

\section{The uniqueness of Hoggar's scheme}\label{sec:Uniqueness}

Recall that $\mathcal{H}$ denotes the association scheme on $256$ vertices arising from Hoggar's 
$64$ equiangular lines in $\mathbb{C}^8$, as described in Section \ref{sec:intro}, and that $\mathcal{H}$ has the eigenmatrix given in \eqref{eq:2}.
The main result of this section is that $\mathcal{H}$ is characterised by its eigenmatrix:
\begin{thm}\label{thm:1}
An association scheme with the eigenmatrix \eqref{eq:2} is isomorphic to $\mathcal{H}$.
\end{thm}

We give the proof of Theorem~\ref{thm:1} now, but note that it requires Lemma~\ref{lem:1} and Theorem~\ref{thm:2}, which are provided subsequently and are of independent interest.
The idea of the proof of Theorem~
\ref{thm:2} is similar to 
\cite{BBB2004,GS2024}.

\begin{proof}[Proof of 
Theorem~\ref{thm:1}]
The eigenmatrix in \eqref{eq:2} satisfies the conditions of Theorem \ref{thm:thinreg_basic} where 
$\mathbb{T}$ is the character table of $\mathbb{Z}_4$, $D_+ = \text{diag}(63, 21, -9, 21)$ and $D_- = \text{diag}(-1, -3, 7, -3)$, so the thin relations of $\cX$ act regularly on the thick relations. By Lemma \ref{lem:1} $\cX$ is triply regular, and hence for any vertex $x$ the neighbourhood $R_4(x)$ 
induces an association scheme $\cY$
with eigenmatrix \eqref{eq:HexFissionP}.
By Theorem \ref{thm:decomp} it follows that $\cX = \cY \widehat\otimes Z_4$. Since $\cY$ is unique by Theorem~\ref{thm:2}, it follows that $\cX$ is unique.
\end{proof}

\begin{lem}\label{lem:1}   
An association scheme with the eigenmatrix in \eqref{eq:2} is triply regular. Moreover, any local scheme (with respect to a thick relation) has eigenmatrix \eqref{eq:HexFissionP}. 
\end{lem}

\begin{proof}
In view of \eqref{eqn:triple2}, we are only interested in the 
triple intersection numbers $[i\ j\ k]$ with $i,j,k\in \{1,\ldots,7\}$.
Then, given a triple of pairwise distinct points $u,v,w\in \Omega$ 
(more precisely, the relations $U,V,W$ containing the corresponding points), 
Proposition \ref{prop:qijk=0} together with \eqref{eqn:triple} 
yields a linear system of equations in $7^3$ unknown 
triple intersection numbers with respect to $u,v,w$. 
Independent of the choice of relations $U, V, W$,
we obtain by computer a linear system of full rank with the same coefficient matrix. Thus there is a unique solution.
This shows that $\mathcal{X}$ is triply regular. 
Moreover, this means that the triple intersection numbers are the same as those of $\mathcal{H}$. After computing the eigenmatrix of the local scheme from its intersection numbers, we obtain \eqref{eq:HexFissionP}.
\end{proof}

If an association scheme with eigenmatrix \eqref{eq:HexFissionP} has relations $S_0, \ldots, S_4$, then 
$S_1$ defines a distance-regular graph $\Gamma$ of diameter $3$ with intersection array $\{6,4,4;1,1,3\}$, that is a generalized hexagon of order $(2,2)$. Moreover, $\Gamma$ is the symmetrisation of the association scheme, and $\Gamma_2 = S_3$ and $\Gamma_3 = S_2 \cup S_4$ are the 
distance 2 and 3 relations of $\Gamma$, respectively.

\begin{thm}\label{thm:2}
An association scheme with the eigenmatrix in \eqref{eq:HexFissionP} is unique up to isomorphism. In particular, it is the $4$-class fission scheme of the distance-regular collinearity graph of the split Cayley generalised hexagon $H(2)$ described in Example~\ref{ex:hex}.
\end{thm}

\begin{proof}
Let $\cX$ be an association scheme with eigenmatrix \eqref{eq:HexFissionP} and $\Gamma$ the distance-regular graph arising from the first relation.
Fix the following seven vertices 
$\boldsymbol{x}_1,\ldots,\boldsymbol{x}_7$ with $\boldsymbol{x}_2,\boldsymbol{x}_3,\boldsymbol{x}_4\in \Gamma_1(\boldsymbol{x}_1)$, $\boldsymbol{x}_5,\boldsymbol{x}_6,\boldsymbol{x}_7\in \Gamma_2(\boldsymbol{x}_1)$, $\boldsymbol{x}_5\in \Gamma_2(\boldsymbol{x}_2),\boldsymbol{x}_6\in \Gamma_2(\boldsymbol{x}_4),\boldsymbol{x}_7\in \Gamma_1(\boldsymbol{x}_3)$. 

Computing the second eigenmatrix of $\cX$ gives
\[
\left[
\begin{array}{ccccc}
 1 & 21 & 27 & 7 & 7 \\
 1 & \frac{21}{2} & -\frac{9}{2} &  -\frac{7}{2} & -\frac{7}{2} \\
 1 & -\frac{21}{8} & \frac{27}{8} &  \zeta & \overline{\zeta} \\
 1 & 0 & -\frac{9}{2} & \frac{7}{4} & \frac{7}{4} \\
 1 & -\frac{21}{8} & \frac{27}{8} & \overline{\zeta} & \zeta
 \end{array}
\right], 
\]
where $\zeta = -\frac{7+21i}{8}$. 
Consider the spherical embedding 
of these vertices into $\mathbb{C}^7$ by the primitive idempotent $\Emat_{3}=\frac{1}{63}(7\Amat_0 - \frac{7}{2}\Amat_1 + \zeta \Amat_2 + \frac{7}{4}\Amat_3 + \overline{\zeta} \Amat_4)$ of $\cX$ (see \eqref{eq:HexFissionP} for ordering of the adjacency matrices). 
Then the Gram matrix of the vertices $\boldsymbol{x}_1,\ldots,\boldsymbol{x}_7$ is of the form: 
\begin{align*}
    G=\left[
\begin{array}{ccccccc}
 1 & \beta_1 & \beta_1 & \beta_1 & \beta_2 & \beta_2 & \beta_2 \\
 \beta_1 & 1 & \beta_2 & \beta_2 & \beta_2 & \gamma_{1} & \gamma_{2} \\
 \beta_1 & \beta_2 & 1 & \beta_2 & \gamma_{3} & \gamma_{4} & \beta_1 \\
 \beta_1 & \beta_2 & \beta_2 & 1 & \gamma_{5} & \beta_2 & \gamma_{6} \\
 \beta_2 & \beta_2 & \overline{\gamma_{3}} & \overline{\gamma_{5}} & 1 & \gamma_{7} & \gamma_{8} \\
 \beta_2 & \overline{\gamma_{1}} & \overline{\gamma_{4}} & \beta_2 & \overline{\gamma_{7}} & 1 & \gamma_{9} \\
 \beta_2 & \overline{\gamma_{2}} & \beta_1 & \overline{\gamma_{6}} & \overline{\gamma_{8}} & \overline{\gamma_{9}} & 1 \\
\end{array}
\right],
\end{align*}
where $\beta_1=-1/2$, $\beta_2=1/4$, 
$\beta_3=\overline{\zeta}/7$, $\beta_4={\zeta}/7$
and $\gamma_1,\ldots,\gamma_6\in\{\beta_3,\beta_4\},\gamma_7,\gamma_8,\gamma_9\in\{\beta_2,\beta_3,\beta_4\}$. 

By computer we find that there are exactly 16 possibilities for $G$ to be positive semidefinite (and in fact they turn out to also be positive definite),  but only 10 up to permutational equivalence.
Let $G_1,\ldots,G_{10}$ be these Gram matrices.
For each $G_i$, consider the Cholesky decomposition: $G_{i}=L_i L_i^*$ where $L_i$ is a lower triangular matrix with non-zero diagonal. 
The image of the spherical embedding of the remaining 56 vertices of the scheme must have inner products in $\beta_1,\ldots,\beta_4$ with all the row vectors of $L_i$. 
Therefore the 56 vectors 
on the unit sphere $S_{\mathbb{C}}^6$
must be in the following set: 
\[
Z_i=\{\boldsymbol{u}^\top\in S_{\mathbb{C}}^6 \mid L_i\boldsymbol{u}^* \in \{\beta_1,\beta_2,\beta_3,\beta_4\}^{7}\}.
\]
Note that $|Z_i|=106$ for each $i$.  
Next, we want to find the vectors in $Z_i$ with inner products in $\beta_1,\ldots,\beta_4$. 
Consider a graph with vertex set $Z_i$ and edge set $E_i$  defined by $\{\boldsymbol{y},\boldsymbol{y}'\}\in E_i$ if and only if 
$\langle \boldsymbol{y},\boldsymbol{y}'\rangle\in
\{\beta_1,\beta_2,\beta_3,\beta_4\}$. 
By computer, we find that for each $i\in\{1,\ldots,10\}$, there exist exactly two cliques, say $C_{i,1},C_{i,2}$, of order $56$ in the graph $(Z_i,E_i)$. 
For $i=1,\ldots,10$ and $j=1,2$, the Gram matrices of the row vectors of $L_i$ and the vectors in $C_{i,j}$ are permutationally 
equivalent.
 
Thus, we may assume that the $63$ vertices 
are represented by the row vectors from $L_1$, $C_{1,1}$, and we can directly verify that, together with the binary relations defined by their inner products, they form an association scheme of $4$ classes with the same parameters 
as the $4$-class fission scheme of the distance-regular collinearity graph of the split Cayley generalised hexagon $H(2)$ described in Example \ref{ex:hex}.
This completes the proof of Theorem \ref{thm:2}. 
\end{proof}

\subsection*{Acknowledgements} 
We would like to thank Yuefeng Yang for pointing out some typos in an earlier version of this manuscript.
The research of Sho Suda is supported by JSPS KAKENHI Grant Number 22K03410.
Jesse Lansdown is a JSPS International Research Fellow and additionally was supported by a 2023 University of Canterbury Faculty of Engineering Strategic Research Grant and by JSPS KAKHENHI Grant Number 24KF0099.

\bibliographystyle{amsplain}

\providecommand{\bysame}{\leavevmode\hbox to3em{\hrulefill}\thinspace}
\providecommand{\MR}{\relax\ifhmode\unskip\space\fi MR }
\providecommand{\MRhref}[2]{%
  \href{http://www.ams.org/mathscinet-getitem?mr=#1}{#2}
}
\providecommand{\href}[2]{#2}

\end{document}